\documentclass[11pt,a4paper]{article}                               
\usepackage[T1]{fontenc}                   
\usepackage[utf8]{inputenc}                 
\usepackage{graphicx}                      % 
\usepackage{caption}        
\usepackage{mathrsfs}          %                  
\usepackage[top=.8in,bottom=1in,left=1in,right=1in]{geometry}
\usepackage{verbatim}
\usepackage{color}

\usepackage{picture}
\usepackage{amsmath,amsthm,amsfonts,amssymb}
\usepackage{accents}
\usepackage{pgf,tikz}
\usetikzlibrary{arrows}
\newtheorem{thm}{Theorem}[section]
\newtheorem{lem}[thm]{Lemma}

\newtheorem{prop}[thm]{Proposition}
\newtheorem{defn}{Definition}
\theoremstyle{definition}
\newtheorem{rem}[thm]{Remark}

\theoremstyle{remark}

\newcommand{\ubar}[1]{\underaccent{\bar}{#1}}

\newcommand{\haus}{\mathcal{H}^{n-1}}
\newcommand{\ds}{\displaystyle}

\newcommand{\R}{\mathbb{R}}

\newcommand{\W}{\mathcal W}
\newcommand{\de}{\partial}

\DeclareMathOperator{\diam}{diam}

\title{Anisotropic Isoperimetric Inequalities involving Boundary Momentum, Perimeter and Volume} 
\author{
	Gloria Paoli%
	\thanks{
		Universit\`a degli studi di Napoli Federico II, Dipartimento di Matematica e Applicazioni ``R. Caccioppoli'', Via Cintia, Monte S. Angelo - 80126 Napoli, Italia. Email: gloria.paoli@unina.it }
	{, }
	Leonardo Trani%
	\thanks{Universit\`a degli Studi di Napoli Federico II, Dipartimento di Matematica e Applicazioni ``R. Caccioppoli'', Via Cintia, Monte S. Angelo - 80126 Napoli, Italia. Email: leonardo.trani@unina.it}
}

\begin{document}
	\maketitle
\begin{abstract}
%This template helps you to create a properly formatted \LaTeX\ manuscript.
We consider a scale invariant functional involving the anisotropic $p$-momentum, the anisotropic perimeter and the volume.
We show that the Wulff shape, associated with the Finsler norm $F$ considered and centered at the origin, is the unique minimizer of the anisotropic functional taken into consideration among all bounded convex sets.
\end{abstract}



\section{Introduction}
Let $\Omega$ be a bounded, connected, open subset of $\mathbb{R}^n$, $n\geq2$, with Lipschitz boundary; its Steklov eigenvalues related to the Laplacian are the real numbers $\sigma\geq 0$ such that   
\begin{equation*}
\begin{cases}
-\Delta u=0 &\mbox{in}\ \Omega\\[.2cm]
\dfrac{\partial u}{\partial \nu}=\sigma u &\mbox{on}\ \de\Omega
\end{cases}
\end{equation*}
admits non trivial $H^1(\Omega)$ solutions. In particular, the first non trivial Steklov eigenvalue of $\Omega$ is characterised by the following expression (see \cite{h}):
$$\sigma_1(\Omega)=\min\left\{\dfrac{\ds\int_{\Omega}|\nabla v|^2\;dx}{\ds\int_{\de\Omega}v^2\;d\haus(x)} \; :\; v\in H^1(\Omega)\setminus\{0\},\;\int_{\de\Omega}v\;d\haus=0    \right\}.$$ 
If we consider the problem of maximizing $\sigma_1$ under volume constraint, the Brock-Weinstock inequality tells us that the unique solutions to this problem are given by balls, for more details see \cite{br}. In scaling invariant form, the inequality has the form:
$$ \sigma_1(\Omega)\leq \left( \dfrac{\omega_n}{V(\Omega)}  \right)^{\frac{1}{n}},$$
where $\omega_n$ is the Lebesgue measure of the $n$-dimensional ball of radius $1$ and $V(\Omega)$ is the Lebesgue measure of the set. Here equality holds if and only if $\Omega$ is a ball. We point out that in dimension $n=2$ there is a stronger result, the so called Weinstock inequality, that states that disks are still maximizers among all simply connected sets of given perimeter.
\newline
The main result of \cite{bfnt} is the following theorem. We denote by $P(\Omega)$ the perimeter of the set $\Omega$ and by $B$ the $n$-dimensional unit ball.
\begin{prop}{\bf \cite[Theorem 3.1]{bfnt}}
	Let $\Omega$ be a bounded, open  convex set of $\mathbb{R}^n$. Then 
	\begin{equation}\label{steklov_convex}
	\sigma_1(\Omega)\left( P(\Omega)\right)^{\frac{1}{n-1}} \leq \sigma_1(B)\left( P(B)\right)^{\frac{1}{n-1}}
	\end{equation}
	and there is  equality only if $\Omega$ is a ball centered at the origin.
\end{prop}
The core of the strategy to prove \eqref{steklov_convex} is to show that the  following isoperimetric inequality holds true among all bounded, open and convex sets of $\mathbb{R}^n$:
\begin{equation}\label{main_iso}
\dfrac{\ds\int_{\de\Omega}|x|^2\;d\haus(x) }{\left(\ds\int_{\de\Omega}d\haus(x)\right)\;V(\Omega)^{\frac{2}{n}}}\geq\omega_n^{\frac{-2}{n}}
\end{equation}
and equality holds if and  only if $\Omega$ is a ball centered at the origin. In order to prove the latter inequality, the authors use the notions of shape derivative and inverse mean curvature flow.
\newline
% A smooth boundary $\de\Omega$ of an open set $\Omega=\Omega(0)$ flows by inverse mean curvature if exists a time dependent family $\left(\de\Omega(t)\right)_{t\in[0,T)}$ of smooth boundaries such that the normal velocity at any point $x\in\de\Omega(t)$ is equal to the inverse of the mean curvature of $\de\Omega(t)$ at $x$.

In the present paper we prove an anisotropic generalization of the inequality \eqref{main_iso}. So, this work is devoted to the study of a particular optimization problem, where the scale invariant functional to be optimized is 
$$\mathcal{F}(\Omega)=\dfrac{\ds\int_{\partial \Omega} [F^o(x) ]^p\;F(\nu_{\partial\Omega}(x))\;d\mathcal{H}^{n-1}(x) }{\left(\ds\int_{\partial \Omega}F(\nu_{\partial\Omega}(x))\;d\mathcal{H}^{n-1}(x) \right)V(\Omega)^{\frac{p}{n}}},$$
where $p>1$, $\nu_{\de\Omega}$ is the outward unit normal to $\de\Omega$, $F$ is a Finsler norm and $F^o$ is its dual norm  (see Section $2$ for definitions). As we can see, we are considering a functional involving a particular weighted $p$-momentum and area measure. Physically, the term $F(\nu_{\de\Omega})$ plays the role of a surface tension of a flat surface whose normal is $\nu_{\de\Omega}$ and can be considered as the anisotropy.
\newline
For the ease of the reader, it is worth mentioning here the Betta-Brock-Mercaldo-Posteraro weighted isoperimetric inequality, proved in \cite{bbmp}, that is
\begin{equation}\label{betta}
\int_{\de\Omega} a(|x|)\;d\haus(x)\geq\int_{\de B} a(|x|)\;d\haus(x),
\end{equation}
where $\Omega\subseteq\mathbb{R}^n$ is bounded with smooth boundary, $B$ is a ball centered at the origin having the same measure of $\Omega$ and $a$ is a given non negative function satisfying suitable assumptions.

By adapting the arguments of proof in \cite{bfnt}, we are able to prove the following theorem, that is the non linear counterpart of \eqref{main_iso}. We recall that the Wulff shape of radius $r$ centered at the point $x_0$ is defined as
$$\mathcal W_r(x_0) = \{  \xi \in  \R^n \colon F^o(\xi-x_0)< r \}.$$
We denote by $\kappa_n$ the volume of the Wulff shape of radius $1$ centered at the origin.
\begin{thm}[Main Result]
	Let $\Omega$ be a bounded, open convex set of $\mathbb{R}^n$. Then 
	$$\mathcal{F}(\Omega)\geq \kappa_n^{-\frac{p}{n}}, $$
	and   equality holds only for Wulff shapes centered at the origin.
\end{thm}
Since we are concerned with investigating the first variation of $ \mathcal{F}(\Omega)$ from the point of view of Finsler geometry, we need to use an integration by part formula on manifolds.
Moreover, thanks to an approximation argument, we can compute it  assuming the smoothness of the boundary of the sets. 
\newline
A fundamental tool that we use is the inverse anisotropic mean curvature flow (we refer to \cite{x} and \cite{dpgx} for details). Roughly speaking, the smooth boundary $\de\Omega$ of an open set $\Omega=\Omega(0)$ flows by anisotropic inverse mean curvature if there exists a time dependent family $\left(\de\Omega(t)\right)_{t\in[0,T)}$ of smooth boundaries such that the  anisotropic normal velocity at any point $x\in\de\Omega(t)$ is equal to the inverse of the  anisotropic mean curvature of $\de\Omega(t)$ at $x$.
We will give the exact definition of anisotropic mean curvature (that we denote by $H_F$) and of anisotropic normal in Section $2.4$. 
We make also use of the following anisotropic version of the Heintze-Karcher inequality
$$ \int_{\de\Omega}\dfrac{F(\nu_{\de\Omega})}{H_F}d\haus\geq \dfrac{n}{n-1}V(\Omega),$$
see \cite{r} for the Euclidean case and \cite{xz} for its anisotropic analogous.

The results of the present paper are mainly motivated by  possible applications to the Steklov spectrum problem for the pseudo $p$-Laplacian (we point out \cite{bf} as a reference).

The structure of the work is the following. In sections $2.1$ and $2.2$ we give some notation and we state the main hypothesis on the norm $F$. In section $2.3$ we recall some basic definitions and some properties of the Euclidean perimeter. Section $2.4$ is devoted to the study of the anisotropic case; we give here some  definitions and provide a framework in order to make the exposition as self contained as possible.
Finally, in the last chapter, we prove the main theorem.

\section{Preliminaries}

\subsection{Notation}
In the following we denote by $\langle\cdot,\cdot\rangle$ the standard Euclidean scalar product in $\mathbb{R}^n$ and by $|\cdot|$ the Euclidean norm in $\mathbb{R}^n$, for $n\geq 2$.
We denote with $\mathcal{L}^n$ the Lebesgue measure in $\mathbb{R}^n$ and with $\mathcal{H}^k$, for $k\in [0,n]$, the $k-$dimensional Hausdorff measure in $\mathbb{R}^n$.

If $\Omega\subseteq \mathbb{R}^n$, ${\rm Lip}(\de \Omega)$ (resp. ${\rm Lip}(\de \Omega; \mathbb{R}^n)$) is the class of all Lipschitz functions (resp. vector fields) defined on $\de\Omega$.
If $\Omega$ has Lipschitz boundary, for $\mathcal{H}^{n-1}-$ almost every $x\in\partial\Omega$, we denote by $\nu_{\de\Omega}(x)$ the outward unit Euclidean normal to $\partial \Omega$ at $x$ and by $T_x(\de\Omega)$ the tangent hyperplane to $\de\Omega$ at $x$.

\subsection{Finsler norm}

Let $F$ be  a Finsler norm on  $\R^{n}$, i.e. $F$ is a convex non negative function such that 
\begin{equation}
\label{eq:omo}
F(t\xi)=|t|F(\xi), \quad t\in \R,\,\xi \in  \R^{n}, 
\end{equation}
and 
\begin{equation}
\label{eq:lin}
a|\xi| \le F(\xi),\quad \xi \in  \R^{n},
\end{equation}
for some constant $a>0$. The hypotheses on $F$ imply that there exists $b\ge a$ such that
\begin{equation*}
\label{upb}
F(\xi)\le b |\xi|,\quad \xi \in  \R^{n}.
\end{equation*}
Moreover, throughout the paper, we will assume that $F\in C^{2}(\mathbb R^{n}\setminus \{0\})$, and
\begin{equation*}
\label{strong}
[F^{p}]_{\xi\xi}(\xi)\text{ is positive definite in } \R^{n}\setminus\{0\},
\end{equation*}
with $1<p<+\infty$. 
The polar function $F^o\colon \R^n \rightarrow [0,+\infty[$ 
of $F$ is defined as
\begin{equation*}
F^o(v)=\sup_{\xi \ne 0} \frac{\langle \xi, v\rangle}{F(\xi)}. 
\end{equation*}
It is easy to verify that also $F^o$ is a convex function
which satisfies properties \eqref{eq:omo} and
\eqref{eq:lin}. $F$ and $F^o$ are usually called dual Finsler norms. Furthermore, 
\begin{equation*}
F(v)=\sup_{\xi \ne 0} \frac{\langle \xi, v\rangle}{F^o(\xi)}.
\end{equation*}
The above property implies the following anisotropic version of the Cauchy Schwartz inequality
\begin{equation*}
\label{imp}
|\langle \xi, \eta\rangle| \le F(\xi) F^{o}(\eta), \qquad \forall \xi, \eta \in  \R^{n}.
\end{equation*}

We denote by 
$$\mathcal W = \{  \xi \in  \R^n \colon F^o(\xi)< 1 \},$$
the  Wulff shape centered at the origin and we put
$\kappa_n=V(\mathcal W)$. Moreover, we assume that $\mathcal{W}$ is uniformly convex, i.e. there exists a constant $c>0$ such that the principal curvatures $\kappa_i(\mathcal{W})> c$, for every $i=1,\cdots, n-1$. 

We conclude this paragraph reporting the following properties of $F$ and $F^o$ (see for istance \cite{bpa}):
\begin{gather*}
\label{prima}
\langle \nabla F(\xi) , \xi \rangle= F(\xi), \quad  \langle\nabla F^{o} (\xi), \xi \rangle
= F^{o}(\xi),\qquad \forall \xi \in
\R^n\setminus \{0\}
\\
\label{seconda} F(  \nabla F^o(\xi))=F^o( \nabla F(\xi))=1,\quad \forall \xi \in
\R^n\setminus \{0\}, 
\\
\label{terza} 
F^o(\xi)   \nabla F( \nabla F^o(\xi) ) = F(\xi) 
\nabla F^o\left(  \nabla F(\xi) \right) = \xi\qquad \forall \xi \in
\R^n\setminus \{0\}. 
\end{gather*}

\subsection{The first variation of euclidean perimeter}
For the content of  this section we refer, for instance, to Chapter $2$ in  \cite{b}  and \cite{m}  (in particular Section $17.3$). We start from recalling  the definition of tangential gradient.
\begin{defn}
	Let $\Omega$ be an open, bounded subset of $\mathbb{R}^n$ with $C^{\infty}$ boundary and let $u:\mathbb{R}^n\rightarrow\mathbb{R}$ be a Lipschitz function. We can define the tangential gradient of $u$ for almost every $x\in \de\Omega$ as follows:
	$$ \nabla^{\de\Omega}u(x)=\nabla u(x)-\langle \nabla u(x),\nu_{\de\Omega}(x)\rangle\nu_{\de\Omega}(x),$$
	whenever $\nabla u$  exists at $x$.
\end{defn}

If we consider a vector field $T\in C^1_c(\mathbb{R}^n;\mathbb{R}^n)$, we can also define the tangential divergence of $T$ on $\de\Omega$ by the formula
$$ {\rm div}^{\de\Omega}T={\rm div}T-\langle \nabla T\nu_{\de\Omega},\nu_{\de\Omega}\rangle.$$

The following theorem is an extention to hypersurfaces in $\mathbb{R}^n$ of Gauss-Green theorem (see in \cite{m} Theorem $11.8$ combined with Remark $17.6$).
\begin{thm}
	Let  $\Omega$ be a subset of $\mathbb{R}^n$ with $C^2$ boundary.
	Then there exists a continuous scalar function $H_{\de\Omega}:\de\Omega\rightarrow\mathbb{R}$ such that for every $\varphi\in C^1_c(\mathbb{R}^n)$
	$$ \int_{\de\Omega}\nabla^{\de\Omega}\varphi(x)\;d\mathcal{H}^{n-1}(x)=\int_{\de\Omega}\varphi(x){H}_{\de\Omega}(x)\nu_{\de\Omega}(x)\;d\mathcal{H}^{n-1}(x).$$
\end{thm}

The scalar function $H_{\de\Omega}:\de\Omega\rightarrow\mathbb{R}$  is the so-called mean curvature. 
\begin{rem}
	Using the definition of tangential divergence, the Gauss-Green theorem can be reformulated in the following way:
	$$ \int_{\de\Omega}{\rm div}^{\de\Omega}T(x)\;d\mathcal{H}^{n-1}(x)=\int_{\de\Omega}H_{\de\Omega}(x)\langle T(x),\nu_{\de\Omega}(x)\rangle\;d\mathcal{H}^{n-1}(x),$$
	for every $T\in C^1_c(\mathbb{R}^n;\mathbb{R}^n)$.
\end{rem}
A $1-$parameter family of diffeomorphisms of $\mathbb{R}^n$ is a smooth function 
$$ (x,t)\in\mathbb{R}^n\times(-\epsilon,\epsilon)\mapsto \phi(x,t),$$ 
for $\epsilon>0$ such that, for each fixed $|t|<\epsilon$, $\phi(\cdot,t)$ is a diffeomorphism. We consider here a particular class of $1-$parameter family of diffeomorphisms such that
$ \phi(x,t)=x+t T(x)+O(t^2)$, with $T\in C^1_c(\mathbb{R}^n;\mathbb{R}^n)$. In \cite{m}  (Theorem $17.5$) the following theorem is proved.
\begin{thm}
	Let	$\Omega$ be a  bounded, open set of $\mathbb{R}^n$ with $C^{\infty}$ boundary and let  $\{\phi(\cdot,t)\}_{|t|<\epsilon}$ be a $1-$parameter family of diffeomorphisms as previously defined. We denote by $\Omega(t)$ the image of $\Omega$ through $\phi(\cdot,t)$. Then,
	$$ P(\Omega(t))=P(\Omega)+t\int_{\de\Omega} {\rm div}^{\de\Omega}T(x)\;d\mathcal{H}^{n-1}(x)+o(t).$$
\end{thm}
Using now the Gauss-Green theorem and this last theorem, we obtain the following expression for the first variation of the perimeter of an open set with $C^{\infty} $ boundary:
$$\dfrac{d}{dt} P(\Omega(t))\arrowvert_{t=0}=\int_{\de\Omega}H_{\de\Omega}(x)\langle T(x),\nu_{\de\Omega}(x)\rangle \;d\mathcal{H}^{n-1}(x).$$
\subsection{ The first variation of anisotropic perimeter }
Let  $\Omega$ be a bounded open convex set of $\mathbb{R}^n$; in the following we are fixing a Finsler norm $F$. 

\begin{defn}
	Let $\Omega$ be a bounded open subset of $\R^n$ with Lipschitz boundary, the anisotropic perimeter of $\Omega$ is defined as 
	\[
	P_F(\Omega)=\displaystyle \int_{\de \Omega}F(\nu_{\de \Omega}(x)) \, d \mathcal H^{n-1}(x).
	\]
\end{defn}
Clearly, the anisotropic perimeter of  $\Omega$ is finite if and only if the usual Euclidean perimeter of $\Omega$, that we denote by $P(\Omega)$, is finite. Indeed, by the quoted properties of $F$ we obtain that
$$
aP(\Omega) \le P_F(\Omega) \le bP(\Omega).
$$

\noindent Moreover, an isoperimetric inequality is proved for the anisotropic perimeter, see for istance \cite{aflt,bu,dp,dpg,fm}.
\begin{thm}{\bf \cite[Proposition 2.3]{aflt}}
	Let $\Omega$ be a subset of $\mathbb{R}^n$ with finite perimeter. Then 
	$$
	P_F(\Omega) \ge n \kappa_n^\frac{1}{n} {\left\vert \Omega \right\vert}^{1-\frac{1}{n}} 
	$$
	and equality holds if and only if $\Omega$ is homothetic to a Wulff shape.
\end{thm}
We give now the following definitions. 
\begin{defn}
	Let  $\Omega$ be a subset of $\mathbb{R}^n$ with $C^{\infty}$ boundary. At each point of $\de \Omega$ we define 
	the $F$-normal vector $$\nu^F_{\de \Omega}(x)=\nabla F(\nu_{\de \Omega}(x)),$$
	sometimes called the Cahn-Hoffman field.
\end{defn}

\begin{defn}
	Let   $\Omega$ be a subset of $\mathbb{R}^n$ with $C^{\infty}$ boundary. For every $x\in\partial\Omega$, we define 
	the $F$-mean curvature $$H^F_{\de \Omega}(x)= {\rm div}^{\de\Omega}\left(\nu^F_{\de \Omega}(x)\right).$$
\end{defn}

In \cite[Theorem $3.6$]{brn}  we find the computation of the first variation of the anisotropic perimeter.  We report its statement; in the proof is used the first variation of the euclidean perimeter.
\begin{thm}
	Let $\Omega$ be a bounded open subset of $\R^n$ with $C^{\infty}$ boundary. For $t \in\mathbb{R}$, let $\phi(\cdot,t):\mathbb{R}^n\rightarrow\mathbb{R}^n$ be a family of diffeomorphisms such that $\phi(\cdot,0)=Id$ and $\phi(\cdot,t)-Id$ has compact support in $\mathbb{R}^n$, for $t$ in a neighborhood of $0$. Set $\Omega(t)$ the image of $\Omega$ through $\phi(\cdot,t)$.
	Then
	\begin{equation}\label{anis_per}
	\frac{d}{d t}P_{F}(\Omega(t)) \arrowvert_{t=0}=\int_{\de\Omega}H^F_{\de \Omega}(x)\langle\nu_{\de\Omega}(x),g(x)\rangle d\mathcal{H}^{n-1}(x),
	\end{equation}
	where $g(x):=\dfrac{\partial \phi(x,t)}{\partial t}\arrowvert_{t=0}.$
\end{thm}
For more details on this part  the reader is referred to \cite{x} and \cite{brn}.

\section{Isoperimetric inequality for  certain anisotropic functionals}

Let  $\Omega$ be a bounded, open set of $\mathbb{R}^n$ with Lipschitz boundary. Let $ p>1$,  we consider the following scale invariant functional:
$$\mathcal{F}(\Omega)=\dfrac{\ds\int_{\partial \Omega} [F^o(x) ]^p\;F(\nu_{\partial\Omega}(x))\;d\mathcal{H}^{n-1}(x) }{\left[\ds\int_{\partial \Omega}F(\nu_{\partial\Omega}(x))\;d\mathcal{H}^{n-1}(x) \right]V(\Omega)^{\frac{p}{n}}},$$
where $\nu_{\partial\Omega}(x) $ is the unit outer normal at  $x\in\partial\Omega$.
We define the anisotropic $p$-boundary momentum of $\Omega$ as
$$M_{F}(\Omega)=\int_{\partial \Omega} [F^o(x) ]^p\;F(\nu_{\partial\Omega}(x))\;d\mathcal{H}^{n-1}(x).$$
The main result of this article is the following.

\begin{thm}
	Let $\Omega$ be a  bounded, open, convex set of $\mathbb{R}^n$. The following inequality holds true:
	$$\mathcal{F}(\Omega)\geq \kappa_n^{-\frac{p}{n}} $$
	and   equality holds only for Wulff shapes centered at the origin.
\end{thm}

\begin{rem}
	We observe that from this last theorem follows a particular case of \eqref{betta}. If we take $F(x)=|x|$, we obtain
	$$\left( \int_{\de\Omega}|x|^p \;d\haus(x) \right)^n \geq n^n \omega_n^{1-p} V(\Omega)^{n+p-1}. $$
\end{rem}

In what follows we will need the following definitions:
\begin{itemize}
	\item $r^{F}_{\max}(\Omega):=\max\left\{F^o(x)\;|\;x\in\bar{\Omega}\right\} $.
	\item  $x^F_{\max}(\Omega)\in\de\Omega$ is such that
	$F^o(x_{\max}^F(\Omega))=r^F_{\max}(\Omega)$;
	\item the anisotropic $p$-excess function $E_{F}(\Omega) :=(r^F_{\max}(\Omega))^{p-1}-\dfrac{M_{F}(\Omega)}{nV(\Omega)}.$
\end{itemize}
In order to prove our main theorem, we need some intermediate results that we are now going to illustrate. The general way of proceeding is analogous to the one presented in \cite{bfnt}.

\subsection{The first variation of the $p$ momentum in the smooth case}
Let $\Omega$ be a subset of $\mathbb{R}^n$ with $C^{\infty}$ boundary.
We  consider the following transformations:
\begin{equation}\label{change_variables}
\phi(x,t)=x+t\varphi(x)\nu^F_{\partial \Omega}(x),
\end{equation}
where $\phi\in C^{\infty}_c(\mathbb{R}^n)$ and  $\nu^F_{\partial\Omega}(x)=\nabla F(\nu_{\partial\Omega}(x))$ is the anisotropic normal.
We recall that 
$$ \Omega(t):=\{ x+t\varphi(x)\;\nu_{\partial\Omega}^F(x)\;|\;x\in \Omega\}.$$
From \eqref{anis_per}, we have that
\begin{align*}
&\frac{d}{d t}P_F(\Omega(t))\arrowvert_{t=0}=\int_{\partial \Omega}  H^F_{\partial\Omega}(x)\langle \nu_{\partial\Omega}(x),\varphi(x)\nu_{\partial\Omega}^F(x)\rangle\;d\mathcal{H}^{n-1}(x)=\\&=
\int_{\partial \Omega}  H^F_{\partial\Omega}(x)\varphi(x)\langle \nu_{\partial\Omega}(x),\nabla F(\nu_{\partial\Omega}(x))\rangle\;d\mathcal{H}^{n-1}(x)=\int_{\de\Omega}H^F_{\partial\Omega}(x)\varphi(x)F(\nu_{\de\Omega}(x))\;d\mathcal{H}^{n-1}(x),
\end{align*}
where the last equality holds true because of the properties of a Finsler norm. We  recall also the variation of the volume of a set:
$$\frac{d}{d t}V(\Omega(t))\arrowvert_{t=0}=\int_{\partial \Omega } \varphi(x) F(\nu_{\partial\Omega}(x))\;d\mathcal{H}^{n-1}(x).$$

\begin{prop}
	Let $\Omega$ and $\Omega(t)$ be the subsets of $\mathbb{R}^n$ previously defined. Then
	\begin{align*}
	&\frac{d}{d t}M_{F}(\Omega(t))\arrowvert_{t=0}=\\&=p\int_{\partial \Omega}\left(F^o(x)\right)^{p-1}\langle\nabla F^o(x), \varphi(x)\,\nu^F_{\partial\Omega}(x)\rangle F(\nu_{\partial\Omega} (x))\;d\mathcal{H}^{n-1}(x)+\\&+\int_{\partial\Omega  }[F^o(x) ]^p \,F(\nu_{\partial\Omega}(x))\;H^F_{\partial\Omega}(x)\varphi(x)\;d\mathcal{H}^{n-1}(x).
	\end{align*}
\end{prop}
\begin{proof}
	Considering the change of variables given by \eqref{change_variables}, i.e. $y=\phi(x,t)$, we have that 
	\begin{align*}
	&\frac{d}{d t}M_{F}(\Omega(t))\arrowvert_{t=0}=\\&=\int_{\partial \Omega} \dfrac{d}{d t }\left(  \left[ F^o(\phi(x,t))  \right]^p \right)F(\nu_{\partial\Omega}(\phi(x,t)))\;d\mathcal{H}^{n-1}(\phi(x,t))\arrowvert_{t=0}+\\&+\int_{\partial \Omega} \left( F^o(\phi(x,t))\right)^p \frac{d}{dt}\left[F(\nu_{\partial\Omega}(\phi(x,t))) \;d\mathcal{H}^{n-1}(\phi(x,t))\right]\arrowvert_{t=0}.
	\end{align*}
	We observe that
	\begin{align*}
	&\int_{\partial \Omega}	\dfrac{d}{d t }\left(  \left[ F^o(\phi(x,t))  \right]^p \right)F(\nu_{\partial\Omega}(\phi(x,t)))\;d\mathcal{H}^{n-1}(\phi(x,t))\arrowvert_{t=0}\\&=\int_{\partial \Omega} p \left(F^o(\phi(x,t))\right)^{p-1} \langle\nabla F^o(\phi(x,t)) , \varphi(x)\nu^F_{\partial\Omega}(x)\rangle F(\nu_{\partial\Omega}(\phi(x,t))\;d\mathcal{H}^{n-1}(\phi(x,t))\arrowvert_{t=0}.
	\end{align*}
	
	Moreover, from the first variation of the perimeter \eqref{anis_per}, we can say that
	\begin{equation*}
	\frac{d}{dt}\left[F(\nu_{\partial\Omega}(\phi(x,t))) \;d\mathcal{H}^{n-1}(\phi(x,t))\right]\arrowvert_{t=0}=  H^F_{\partial\Omega}(x)\varphi(x)F(\nu_{\partial\Omega}(x)).
	\end{equation*}
	The thesis follows.

\end{proof}

Considering now the derivative of the quotient, we obtain
\begin{align*}
\frac{d}{d t}\mathcal{F}((\Omega(t))\arrowvert_{t=0}&=\\=\dfrac{1}{P_F(\Omega)^2 V(\Omega)^{\frac{p}{n}}}&\left[p\int_{\partial \Omega}\big[\left(F^o(x)\right)^{p-1}\langle\nabla F^o(x), \nu^F_{\partial\Omega}(x)\rangle F(\nu_{\partial\Omega}(x))\right.\\&\left. -\frac{M_{F}(\Omega)))}{n V(\Omega)}F(\nu_{\partial\Omega}(x))\big]\;\varphi(x)\;d\mathcal{H}^{n-1}(x)+\right.
\\&\left.+\int_{\partial\Omega  }\big[(F^o(x) )^p-\frac{M_{F}(\Omega)}{P_F(\Omega)}\big]\; H^F_{\partial\Omega}(x)\;F(\nu_{\partial\Omega}(x))\varphi(x)\;d\mathcal{H}^{n-1}(x)\right].
\end{align*}

Let be $T>0$;  we choose, as in \cite{xz},
$$\varphi(x)=\dfrac{1}{H^F_{\partial\Omega}(x)},$$
and we have that 
$$ \frac{\partial}{\partial t}\phi(x,t)=\dfrac{\nu^F_{\partial\Omega}(x)}{H^F_{\partial\Omega}(x)},$$
for every $t\in[0,T]$.
This one parameter family of diffeomorphisms gives rise to the so called  inverse anisotropic mean curvature flow (IAMCF). Concerning this family of flows,  local and global existence and uniqueness have been studied in \cite{xz,hu,r}.

\begin{rem}
	Let  $\Omega\subseteq\mathbb{R}^n$ be a bounded convex set of class $C^{\infty}$. 
	$\Omega$ is called $F$-mean convex if its anisotropic mean curvature is strictly positive and, in this case, we say that $\Omega\in C_F^{\infty,+}$. In \cite{xz} is proved that, if $\Omega(0)=\Omega\in C_F^{\infty,+}$, then there exists an unique smooth solution $\phi(\cdot,t)$ of the inverse mean curvature flow in $[0,+\infty]$.
	Moreover the surface $\partial\Omega(t)$, for every $t>0$, is the boundary of a smooth convex set in $C_F^{\infty,+}$ that asymptotically converges to a Wulff shape as $t\rightarrow +\infty$.
\end{rem}

Substituting this $\varphi$ in the derivative of the quotient and taking in account the fact that 
$$ 
\int_{\partial\Omega  }\left[(F^o(x) )^p-\frac{M_{F}(\Omega)}{P_F(\Omega)}\right] F(\nu_{\partial\Omega}(x))\;d\mathcal{H}^{n-1}(x)=0,
$$
we obtain 
\begin{align}\label{imp_2}
& \frac{d}{d t}\mathcal{F}((\Omega(t))\arrowvert_{t=0}=\\\nonumber&\dfrac{p}{P_F(\Omega)^2 V(\Omega)^{\frac{p}{n}}}\int_{\partial \Omega}\left[\left(F^o(x)\right)^{p-1}\langle\nabla F^o(x), \nu^F_{\partial\Omega}(x)\rangle F(\nu_{\partial\Omega}(x))-\frac{M_{F}(\Omega)}{n V(\Omega)} F(\nu_{\de\Omega}(x))\right]\dfrac{d\mathcal{H}^{n-1}(x)}{H^F_{\partial\Omega}(x)}=\\\nonumber&= \dfrac{p}{P_F(\Omega)^2 V(\Omega)^{\frac{p}{n}}}\int_{\partial \Omega}\left[\left(F^o(x)\right)^{p-1}\langle\nabla F^o(x), \nu^F_{\partial\Omega}(x)\rangle-\frac{M_{F}(\Omega)}{n V(\Omega)} \right]\dfrac{F(\nu_{\partial\Omega}(x))}{H^F_{\partial\Omega}(x)}\;d\mathcal{H}^{n-1}(x).
\end{align}

\subsection{Existence of minimizers (Step 1)}

\begin{prop}
	There exists a convex set minimizing $\mathcal{F}(\cdot)$.
\end{prop}
\begin{proof}
	Given a convex set $\Omega$, we can take a minimizing sequence $(\Omega_i)_i$, having the same volume of $\Omega$. 
	By Blaschke selection Theorem in \cite[Theorem 1.8.7]{s}, it is enough to show that the  $\Omega_i$'s are all contained in the same Wulff. For the sake of simplicity, we suppose that $V(\Omega_i)=\kappa_n$ and, since any Wulff $\W$ with centered in the origin is such that $\mathcal{F}(\W)=\kappa_n^{-\frac{p}{n}}$, we have that
	\begin{equation*}
	\lim_{i\rightarrow +\infty}\mathcal F(\Omega_i) \le \kappa_n^{-\frac{p}{n}}, 
	\end{equation*}
	and consequently
	\begin{equation*}
	\lim_{i\rightarrow +\infty} \frac{M_{F}(\Omega_i)}{P_F(\Omega_i)} \le 1.
	\end{equation*}
	Arguing by contradiction, if we assume that $\lim_{i\rightarrow +\infty} \diam_F(\Omega_i)=+\infty$, from  convexity follows easily that $\lim_{i\rightarrow +\infty} P_F(\Omega_i) = +\infty$. Thereafter, if $\W_2$ is the Wulff of anisotropic radius $2$ centered at the origin, it is enough to observe that
	\begin{equation*}
	\lim_{i\rightarrow +\infty}
	\frac{\int_{\de\Omega_i\cap\mathcal{W}_2}F(\nu_{\de\Omega}(x))\;d\haus(x)}{\int_{\de\Omega_i\setminus\W_2} F(\nu_{\de\Omega}(x))\;d\haus(x)} = 0
	\end{equation*}
	and 
	\begin{equation*}
	\lim_{i\rightarrow +\infty} \frac{M_{F}(\Omega_i)}{P_F(\Omega_i)} \ge \lim_{i\rightarrow +\infty} \frac{2^p}{1+ \frac{\int_{\de\Omega_i\cap\mathcal{W}_2}F(\nu_{\de\Omega}(x))\;d\haus(x)}{\int_{\de\Omega_i\setminus\W_2} F(\nu_{\de\Omega}(x))\;d\haus(x)}} = 2^p ,
	\end{equation*}
	which gives a contradiction. So the diameters of the $\Omega_i$'s are equibounded. Moreover, arguing as before, we can show that $\Omega_i\cap\mathcal{W}_2\neq\emptyset$ definitely. Therefore we have the claim.  
\end{proof}

\subsection{A minimizer cannot have negative Excess (Step $2$)}
\begin{rem}
	There exist sets with negative anisotropic $p$-Excess. We prove this fact in dimension $2$ and for $p=2$. We consider the elliptic metric $$F(x,y)=\sqrt{\frac{x^2}{a^2}+\frac{y^2}{b^2}}; $$
	we know that its polar is this elliptic norm
	$$F^o(x,y)=\sqrt{a^2x^2+b^2y^2}. $$
	We consider now the following convex domain:
	$$R_{\epsilon}=\left\{(x,y)\in\mathbb{R}^2\; :\;|x|\leq \dfrac{1}{\epsilon},\, \;|y|\leq \epsilon\right\} .$$
	From the computations we obtain that $V(R_{\epsilon})=4$, $r^F_{\max}(R_{\epsilon})=a/\epsilon+O(\epsilon^3)$ and $M_F(R_\epsilon)=(4a^2/3b)(1/\epsilon^3)+4a/\epsilon+O(\epsilon)$.
\end{rem}

\begin{lem}
	Let $\Omega$ be a bounded, open convex set of $\mathbb{R}^n$. Then
	\begin{equation}\label{imp}
	\left(F^o(x)\right)^{p-1}\langle\nabla F^o(x), \nu^F_{\partial\Omega}(x)\rangle-\frac{M_{F}(\Omega)}{n V(\Omega)}\leq  E_F(\Omega).
	\end{equation}
\end{lem}
\begin{proof}
	We observe that $$\langle\nabla F^o(x), \nu^F_{\partial\Omega}(x)\rangle =\langle\nabla F^o(x),  \nabla F(\nu_{\partial\Omega}(x))\rangle  \leq F(\nabla F^o(x)) F^o(\nabla F(\nu_{\partial\Omega}(x)))=1,$$	
	for the properties of the Finsler norm $F$.
\end{proof}
We prove now a fact, that is an analougous of a property holding in the Euclidean case (see Remark $2$ in \cite{bfnt}).
\begin{rem}
	Let $\Omega$ be a bounded, open convex set of $\mathbb{R}^n$. Then
	\begin{equation}\label{sec_ineq}
	\int_{\partial \Omega}\Big[\left(F^o(x)\right)^{p-1}\langle\nabla F^o(x), \nu^F_{\partial\Omega}(x)\rangle-\frac{M_{F}(\Omega)}{n V(\Omega)} \Big]F(\nu_{\partial\Omega}(x))\;d\mathcal{H}^{n-1}(x) \le 0.
	\end{equation}
	%and equality holds if and only id $\Omega$ is a Wulff Shape.
	%\begin{equation*}\label{excess_inequality}
	%	\int_{\partial \Omega}\Big[\left(F^o(x)\right)^{p-1}\langle\nabla F^o(x), \nu^F_{\partial\Omega}(x)\rangle-\frac{M_{F}(\Omega)}{n V(\Omega)} \Big]d\mathcal{H}^{n-1}(x)\leq 0. 
	%\end{equation*}
	
\end{rem}
\begin{proof}
	%We notice that it is enough to prove that
	%\begin{equation}\label{sec_ineq}
	%\int_{\partial \Omega}\Big[\left(F^o(x)\right)^{p-1}\langle\nabla F^o(x), \nu^F_{\partial\Omega}(x)\rangle-\frac{M_{F}(\Omega)}{n V(\Omega)} \Big]F(\nu_{\partial\Omega}(x))\;d\mathcal{H}^{n-1}(x) \le 0,
	%\end{equation}
	%since there exists $\alpha>0$ such that $\alpha\leq F(\nu_{\partial\Omega}(x))$.
	In order to prove \eqref{sec_ineq}, we observe that
	\begin{align*}
	&\int_{\partial \Omega}\Big[\left(F^o(x)\right)^{p-1}\langle\nabla F^o(x), \nu^F_{\partial\Omega}(x)\rangle F(\nu_{\partial\Omega}(x))-\frac{M_{F}(\Omega)}{n V(\Omega)} F(\nu_{\partial\Omega}(x))\Big]d\mathcal{H}^{n-1}(x)	=\\&\int_{\partial \Omega}\Big[\left(F^o(x)\right)^{p-1}\langle\nabla F^o(x), \nabla F(\nu_{\partial\Omega}(x))\rangle F(\nu_{\partial\Omega}(x))-\frac{M_{F}(\Omega)}{n V(\Omega)} F(\nu_{\partial\Omega}(x))\Big]d\mathcal{H}^{n-1}(x)   \\&
	\leq\int_{\partial \Omega}[\left(F^o(x)\right)^{p-1}F(\nu_{\partial\Omega}(x))]\;d\mathcal{H}^{n-1}(x)-\frac{M_{F}(\Omega)}{n V(\Omega)} P_F(\Omega)  \\&
	\leq\int_{\partial \Omega}[\left(F^o(x)\right)^{p-1}F(\nu_{\partial\Omega}(x))]\;d\mathcal{H}^{n-1}(x)-\frac{M_{F}(\Omega)P_F(\Omega) }{\int_{\partial \Omega} F^o(x)F(\nu_{\partial\Omega}(x))\;d\mathcal{H}^{n-1}(x)} 
	\end{align*}
	and the last inequality holds since 
	$$nV(\Omega)=\int_{\partial \Omega} \langle x,\nu_{\partial\Omega}(x)\rangle \;d\mathcal{H}^{n-1}(x)\leq \int_{\partial \Omega} F^o(x)F(\nu_{\partial\Omega}(x))\;d\mathcal{H}^{n-1}(x),$$
	for the properties of the Finsler norms.
	Using now  H\"{o}lder inequality, we obtain
	\begin{align*}
	&\int_{\partial \Omega}\left(F^o(x)\right)^{p-1}F(\nu_{\partial\Omega}(x))\;d\mathcal{H}^{n-1}(x)\\&\leq\left[\int_{\partial \Omega}\left[\left(F^o(x)\right)^{p-1}\right]^{\frac{p}{p-1}}F(\nu_{\partial\Omega}(x))\;d\mathcal{H}^{n-1}(x)\right]^{\frac{p-1}{p}}\; \left(P_F(\Omega)\right)^{\frac{1}{p}}\\&=\left[\int_{\partial \Omega}\left(F^o(x)\right)^{p}F(\nu_{\partial\Omega}(x))\;d\haus \right]^{\frac{p-1}{p}}\; \left(P_F(\Omega)\right)^{\frac{1}{p}}
	\end{align*}
	and
	\begin{equation*}
	\int_{\partial \Omega} F^o(x)F(\nu_{\partial\Omega}(x))\;d\mathcal{H}^{n-1}(x)\leq\left[\int_{\partial \Omega}\left(F^o(x)\right)^{p}F(\nu_{\partial\Omega}(x))\;d\mathcal{H}^{n-1}(x) \right]^{\frac{1}{p}}\; \left(P_F(\Omega)\right)^{\frac{p-1}{p}}.
	\end{equation*}
	Finally, from these last two inequalities follows that
	$$\left(\int_{\partial \Omega}[\left(F^o(x)\right)^{p-1}F(\nu_{\partial\Omega}(x))]\;d\mathcal{H}^{n-1}(x)\right)\left( \int_{\partial \Omega} F^o(x)F(\nu_{\partial\Omega}(x))\;d\mathcal{H}^{n-1}(x)\right)\leq M_{F}(\Omega)P_F(\Omega).$$
	%If $\Omega$ is a Wulff shape, then $(F^o(x))^{p-1}=\left(r^F_{\max}(\Omega)\right)^{p-1}$ for every $x\in\de\Omega$; so we have equality in \eqref{sec_ineq}. We consider now the case when $\Omega$ is not a Wulff shape. Then, there exists $\bar{x}\in \partial\Omega$ such that $(F^o(\bar{x}))^{p-1}<\left(r^F_{\max}(\Omega)\right)^{p-1}$. Since $F^o$ is continous, there exists $C\subseteq\de\Omega$  such that $\bar{x}\in\C$, $\mathcal{H}^{n-1}(C)>0$ and 
	%\begin{equation*}
	%C=\{x\in\de\Omega\;|\; F^o(x)<r^F_{\max}(\Omega)\}.
	%\end{equation*}
	%Thus, we have
	%\begin{multline*}
	%\int_{\partial \Omega}\Big[\left(F^o\right)^{p-1}\langle\nabla F^o, \nu^F_{\partial\Omega}(x)\rangle-\frac{M_{F}(\Omega)}{n V(\Omega)} \Big]F(\nu_{\partial\Omega})\;d\mathcal{H}^{n-1}=\int_{C}\Big[\left(F^o\right)^{p-1}\langle\nabla F^o, \nu^F_{\partial\Omega}\rangle-\frac{M_{F}(\Omega)}{n V(\Omega)} \Big]F(\nu_{\partial\Omega})\;d\mathcal{H}^{n-1}\\+\int_{\partial \Omega\setminus C}\Big[\left(F^o\right)^{p-1}\langle\nabla F^o, \nu^F_{\partial\Omega}\rangle-\frac{M_{F}(\Omega)}{n V(\Omega)} \Big]F(\nu_{\partial\Omega})\;d\mathcal{H}^{n-1} \\< \int_{ C}\Big[\left(r^F_{\max}(\Omega)\right)^{p-1}\langle\nabla F^o, \nu^F_{\partial\Omega}\rangle-\frac{M_{F}(\Omega)}{n V(\Omega)} \Big]F(\nu_{\partial\Omega})\;d\mathcal{H}^{n-1}+\int_{\partial\Omega\setminus C}\Big[\left(r^F_{\max}(\Omega)\right)^{p-1}\langle\nabla F^o, \nu^F_{\partial\Omega}\rangle-\frac{M_{F}(\Omega)}{n V(\Omega)} \Big]F(\nu_{\partial\Omega})\;d\mathcal{H}^{n-1}
	%\end{multline*}
	
\end{proof}

We recall now this lemma  (see \cite{xz}), which will be used in the next proofs. This is the anisotropic version of the Heintze-Karcher inequality, whose proof in the Euclidean case can be found in \cite{r}.

\begin{lem}
	Let  $\Omega$ be a bounded, open convex set of $\mathbb{R}^n$, then
	\begin{equation*}\label{anis_mean_curvature}
	\int_{\partial \Omega} \dfrac{F(\nu_{\partial\Omega}(x))}{H_{\partial \Omega}^F(x)}\;d\mathcal{H}^{n-1}(x)\geq \int_{\partial \mathcal{W}} \dfrac{F(\nu_{\partial\mathcal{W}}(x))}{H_{\partial \mathcal{W}}^F(x)}\;d\mathcal{H}^{n-1}(x)
	\end{equation*}
	where $\mathcal{W}$ is a Wulff such that $V(\mathcal{W})=V(\Omega)$.
\end{lem}

\begin{prop}
	Let  $\Omega$ be a bounded, open convex set of $\mathbb{R}^n$ such that 
	\begin{equation*}
	E_F(\Omega)<0,
	\end{equation*}
	then $\Omega$ is not a minimizer of $\mathcal{F}(\cdot)$.
\end{prop}
\begin{proof}
	We firstly assume that $\Omega\in C_F^{\infty,+}$. Since $E_F(\Omega)\neq 0$, $\Omega$ is not a Wullf shape centered at the origin. Then, from \eqref{imp} and \eqref{imp_2},
	we have
	$$\mathcal{F}'(\Omega)\leq \dfrac{p}{P_F(\Omega)^2 V(\Omega)^{\frac{p}{n}}}E_F(\Omega)\int_{\partial \Omega}\dfrac{d\mathcal{H}^{n-1}(x)}{H^F_{\partial \Omega}(x)}<0.$$ 
	We suppose now that $\Omega\notin C_F^{\infty,+}$ and we assume by contradiction that $\Omega$ minimizes the functional $\mathcal{F}(\cdot)$. We can find a decreasing (in the sense of inclusion) sequence of sets $\left(\Omega_k\right)_{k\in\mathbb{N}}\subset C_F^{\infty,+}$ that converges to $\Omega$ in the Hausdorff sense. We have that 
	$$ \lim\limits_{k\rightarrow+\infty}V(\Omega_k)=V(\Omega);\qquad \lim\limits_{k\rightarrow+\infty}P_F(\Omega_k)=P_F(\Omega);$$
	$$  \lim\limits_{k\rightarrow+\infty}M_F(\Omega_k)=M_F(\Omega);\qquad \lim\limits_{k\rightarrow+\infty}r^F_{\max}(\Omega_k)=r^F_{\max}(\Omega).$$
	We now consider the IAMCF for every $\Omega_k$ and we denote by $\Omega_k(t)$, for $t\geq0$, the family generated in this way. We let $\Omega_k(0)=\Omega_k$. Using Hadamard formula (see \cite{hp}), we obtain:
	\begin{equation*}\label{variation_volume}
	\dfrac{d}{dt}V(\Omega_k(t))=\int_{\partial \Omega_k(t)}\dfrac{F(\nu_{\partial\Omega}(x))}{H^F_{\partial\Omega_k(t)}}d\mathcal{H}^{n-1}(x);
	\end{equation*}
	\begin{equation*}\label{variation_perimeter}
	\dfrac{d}{dt}P_F(\Omega_k(t))=P_F(\Omega_k(t)).
	\end{equation*}
	We have also that
	\begin{equation}\label{variation_radius}
	\dfrac{d}{dt} r^F_{\max}(\Omega_k(t))\leq\dfrac{r^F_{\max}(\Omega_k(t))}{n-1}.
	\end{equation}
	We prove now this last inequality.  
	From definition of $x^F_{\max}(\Omega(t))$  and \eqref{change_variables} in the IAMCF case, we have that
	$$r^F_{\max}(\Omega(t))=F^o(x^F_{\max}(\Omega(t)) );$$
	$$ x^F_{\max}(\Omega(t))=x^F_{\max}(\Omega)+\dfrac{t\nu^F_{\de\Omega}}{H^F_{\de\Omega}(x^F_{\max}(\Omega))}.$$
	Then
	\begin{align*}
	&\frac{d}{dt}r^F_{\max}(\Omega(t))=\frac{d}{dt} F^o(x^F_{\max}(\Omega(t)))=\langle\nabla F^o(x^F_{\max}(\Omega(t)),\dfrac{\nu^F_{\de\Omega}(x^F_{\max}(\Omega))}{H^F_{\de\Omega}(x^F_{\max}(\Omega))}\rangle\leq\\&\leq F(\nabla F^o(x^F_{\max}(\Omega(t)))) F^o(\nu^F_{\de\Omega}(x^F_{\max}(\Omega))) \frac{1}{H^F_{\de\Omega}(x^F_{\max}(\Omega))}=\\&= F(\nabla F^o(x^F_{\max}(\Omega(t)))) F^o(\nabla F(\nu_{\de\Omega}(x^F_{\max}(\Omega)))) \frac{1}{H^F_{\de\Omega}(x^F_{\max}(\Omega))}\leq \\&=\frac{1}{H_F(x^F_{\max}(\Omega))}=\dfrac{r^F_{\max}(\Omega)}{n-1},
	\end{align*}
	since $F$ is a Finsler norm and therefore it is true that $F(\nabla F^o(x))=F^o(\nabla F(x))=1$.
	We can then repeat this last inequality for every $\Omega_k$.
	From \eqref{variation_radius} follows that 
	$$ r^F_{\max}(\Omega_k(t))\leq r^F_{\max}(\Omega_k)e^{\frac{t}{(n-1)}},\;{\rm for}\;t>0.$$ Analogous computations to the ones reported in \cite[Proposition 2.4]{bfnt} lead to a contradiction with the minimality of $\Omega$ and therefore to the thesis.
\end{proof}

\subsection{A minimizer cannot have positive Excess.}
We start observing that there exist sets with positive excess.
\begin{rem}
	We consider the case $n=2$ and $p=2$. The norm that we take into consideration is
	$$F(x,y)=\sqrt{\frac{x^2}{a^2}+\frac{y^2}{b^2}}; $$
	and its polar is:
	$$F^o(x,y)=\sqrt{a^2x^2+b^2y^2}. $$
	We define $$\mathcal{E}_{\epsilon}=\{(x,y)\in\mathbb{R}^2\;|\;a^2(1-\epsilon)^2x^2+b^2(1+\epsilon)^2y^2\}.$$
	We have that $$r^F_{\max}(\mathcal{E}_\epsilon)=1+\epsilon+o(\epsilon)$$ and $$V(R_\epsilon)=\dfrac{\pi}{ab} (1+\epsilon^2+o(\epsilon)).$$
	Computing the second momentum, we find that
	$$M_{F}=\dfrac{2}{ab(1-\epsilon)^2(1+\epsilon)^2}\left(\pi+\epsilon\int_{0}^{\pi}\cos(2t)\;dt\right)+o(\epsilon)=\dfrac{2}{ab(1-\epsilon)^2(1+\epsilon)^2}\left(\pi+o(\epsilon)\right) $$
	and so it results that $E_{F}(\mathcal{E}_\epsilon)=\epsilon+o(\epsilon).$
\end{rem}

 In the following, we will use the notations: $\ubar{0}\in\mathbb{R}^{n-1}$ and $x'=(x_1,\dots, x_{n-1})$.
 \newline
 We consider the halfspace $T_{\epsilon}$ that has outer Euclidean normal pointing in the direction given by the outer Euclidean normal to $\Omega$ in the point $x^F_{\max}(\Omega)$ and intersecting $\Omega$ at a distance $\epsilon$ from $x^F_{\max}(\Omega)$.  We define the  sets:
 $$\Omega_{\epsilon}:=\Omega\cap T_{\epsilon},$$
 $$ A_{\epsilon}:=\partial \Omega_{\epsilon}\cap\partial T_{\epsilon},$$
 $$C_\epsilon=\de \Omega\cap T_\epsilon^c,$$
 where $T^c_{\epsilon}$ is the complement of $T_\epsilon$ in $\mathbb{R}^n$,  
 and we define the following quantitities, that vanish as $\epsilon$ goes to $0$:
 $$\Delta M_F:=M_F(\Omega_{\epsilon}) -M_F(\Omega);$$
 $$\Delta V:=V(\Omega_{\epsilon}) -V(\Omega);$$
 $$\Delta P_F:=P_F(\Omega_{\epsilon}) -P_F(\Omega).$$
 Considering Remark $2.2$ in \cite{dpgp}, we can choose the coordinate in such a way that the $x_n$ axis lies in the direction of the outer normal to $T_\epsilon$ and we denote the coordinates of $x^F_{\max}(\Omega)$ by   $x^F_{\max}(\Omega)=:(x'_0,y_0)\in\mathbb{R}^{n-1}\times \mathbb{R}$. Moreover, we call $A'_{\epsilon}\subseteq\mathbb{R}^{n-1}$ the projection of $A_\epsilon $ onto $\{x_n=0\}$.
 \noindent
 
 Let $g:A'_\epsilon\rightarrow \mathbb{R}$ the convave function describing $C_\epsilon$. Since the class of open and bounded convex set with positive mean curvature is dense in the class of open and bounded convex set, we can assume, in particular, that $\Omega$ is strictly convex and, consequently,  that $g$ is a function of class $C^1(A'_\epsilon)$, for $\epsilon>0$ small enough. Let $h:A'_\epsilon\rightarrow \mathbb{R}$ defined by $h(x')=g(x')-(y_0-\epsilon)$, so $h$ is equal to  $0$ on $\de A'_{\epsilon}$.
 \noindent
 
 We observe that $g:A'_{\epsilon}\rightarrow \mathbb{R}$ is such that for any $x=(x',x_n)\in C_\epsilon$ we have $x_n=g(x')$. 
 We call $G(x):=x_n-g(x')$ and, as a consequence, $C_\epsilon $ is the level set $G(x)=0$; the outer normal to $C_\epsilon$ in a point $x=(x',x_n)\in C_\epsilon$ is given by
 $\nabla G(x')/||\nabla G(x')||,$
 i.e. $$\nu_{C_\epsilon}(x)=\dfrac{(-\nabla g(x'),1)}{\sqrt{1+\nabla g(x')^2}} $$

 Since $\nabla g(x'_0)=\ubar{0}$, we have that 
 \begin{equation}
 -\Delta P_F=\int_{A'_\epsilon}\left[F(-\nabla g(x'),1)-F(\ubar{0},1)\right]d x'.
 \end{equation}
 
 \begin{lem}\label{equalzero}
 	We claim that 
 	\begin{equation*}
 	\int_{A'_{\epsilon}}\langle\nabla_{x'}F(\underline{0},1), -\nabla g(x')\rangle dx'=0
 	\end{equation*}
 \end{lem}
 \begin{proof}
 	Since 
 	\begin{align*}
 	\int_{A'_{\epsilon}}\langle\nabla_{x'}F(\underline{0},1), -\nabla g(x')\rangle dx'=-\sum_{i=1}^{n-1}\int_{A'_\epsilon}\dfrac{\de F}{\de x_i}(\ubar{0},1)\dfrac{\de g}{\de x_i}(x')\;dx',
 	\end{align*}
 	it is enough to prove that, for every $i=1,\dots (n-1)$, $$\int_{A'_\epsilon}\dfrac{\de F}{\de x_i}(\underline{0},1)\dfrac{\de g}{\de x_i}(x')\;dx'=\dfrac{\de F}{\de x_i}(\ubar{0},1)\int_{A'_\epsilon}\dfrac{\de g}{\de x_i}(x')\;dx'=0. $$
 	% We define the following vector field, for $i=1,\cdots,n-1$,  $H_i:\mathbb{R}^{n-1}\rightarrow\mathbb{R}^{n-1}$:
 	% $$H_i(\bar{x})=h(\bar{x})e_i $$
 	%($e_i$  is the vector having all zero coordinates, except the $i$-coordinate equal to $h(\bar{x})$) and we have that $${\rm div} H_i(\bar{x})= \dfrac{\de h}{\de x_i}(\bar{x}).$$
 	Using the divergence theorem and the fact that $h$ is equal to $0$ on  $\de A'_\epsilon $, 
 	\begin{equation*}
 	\int_{A'_\epsilon}	\dfrac{\de g}{\de x_i}(x')\;dx'=\int_{A'_\epsilon}	\dfrac{\de h}{\de x_i}(x')\;dx'=\int_{A'_\epsilon}{\rm div}\left(h(x')e_i\right)\;dx'=\int_{\de A'_\epsilon }\langle h(x')e_i, \nu_{\de A'_\epsilon}(x')\rangle d \mathcal{H}^{n-2}(x')=0,
 	\end{equation*}
 	where $e_i$  is the vector having all zero coordinates, except the $i$-coordinate equal to $1$.
 \end{proof}
 
 %From the convexity inequality, we have:
 % \begin{equation}\label{convexity}
 % F(-\nabla g(y),1)-F(\ubar{0},1)\geq \langle\nabla_{\bar{x}}F(\ubar{0},1),-\nabla g(y)\rangle
 %\end{equation}

 \begin{lem}
 	There exists a positive constant $C(\Omega)$ such that for all $\epsilon>0$ small enough, we have that
 	\begin{equation}\label{inequality}
 	|\Delta V|\leq C(\Omega) |\Delta P_F|.
 	\end{equation}
 	
 \end{lem}
 \begin{proof}
 	
 	There exists a Wulff shape centered in the origin, that we denote with $\mathcal{W}_{\max}$,  that contains $\Omega$ and  that  it is tangent to $\Omega$ in the point $x^F_{\max}=(x'_0,y_0)$, with $x'_0\in\mathbb{R}^{n-1}$ and $y_0\in\mathbb{R}$. Moreover, since $\mathcal{W}$ is uniformly convex, there exists a ball $\bar{B}$ that contains $\mathcal{W}_{\max}$ and that  is tangent to $\mathcal{W}_{\max}$ in $x^F_{\max}(\Omega)$. Let $c>0$ be the positive constant such that , for all $i=1,\cdots, n-1$, $\kappa_i(\mathcal{W})> c$, with $\kappa_i(\mathcal{W})$ principal curvature of $\mathcal{W}$.
 	If we denote by $\bar{R}$  the radius of $\bar{B}$, that is centered at a point $(x'_0,y_c)\in\mathbb{R}^{n-1}\times\mathbb{R}$, we have that $\bar{R}=r^F_{\max}(\Omega)/c$.

 	We have that $A_\epsilon\subseteq \bar{B}\cap \partial T_\epsilon$ and we  denote by $\tilde{R}$ the radius of the  $(n-1)$-dimensional ball $ \bar{B}\cap \partial T_\epsilon$ . Now, we have that
 	\begin{equation}\label{diam}
 	{\rm diam}(A_\epsilon)\leq {\rm diam}(\bar{B}\cap \partial T_\epsilon)=2\tilde{R}\leq 2 \sqrt{2\epsilon \bar{R}}.
 	\end{equation}
 	We observe that 
 	\begin{equation}\label{deltaV}
 	-\Delta V=\int_{A'_\epsilon} h(x')dx' \geq \epsilon \dfrac{ \mathcal{L}^{n-1}(A'_\epsilon)}{n}.
 	\end{equation}
 	Using \eqref{deltaV}, \eqref{diam} and the Sobolev Poincar\'{e} inequality 
 	\begin{align*}
 	&-\Delta V=\int_{A'_\epsilon} h(x')\;dx'\leq \left(\int_{A'_\epsilon} h(x')\;dx'\right)^2\dfrac{n}{\epsilon  \mathcal{L}^{n-1}(A'_\epsilon)}\leq\\&\leq C(n)\dfrac{\left(\mathcal{L}^{n-1}(A'_\epsilon)\right)^{2/(n-1)}}{\epsilon}\int_{A'_\epsilon}||\nabla h||^2 dx'\leq C(n) 2 \bar{R}\left(\omega_{n-1}\right)^{2/(n-1)}\int_{A'_\epsilon}||\nabla h||^2 dx'
 	\end{align*}
 	
 	We now consider the function, $x'\in\mathbb{R}^{n-1}\rightarrow F(x',1)$. Using the Taylor expansion with the Lagrange reminder:
 	\begin{align*}
 	& F(-\nabla g(x'),1)-F(\ubar{0},1)=\langle\nabla_{x'}F(\ubar{0},1),-\nabla g(x')\rangle+\frac{1}{2}(-\nabla g(x'))^TD^2F(\tilde{x_y},1)(-\nabla g(y))\geq\\&\geq \langle\nabla_{x'}F(\ubar{0},1),-\nabla g(x')\rangle+ c ||\nabla g(x')||^2.
 	\end{align*}
 	Integrating the last chain of inequalities and using  the result in Lemma \ref{equalzero}, we can conclude
 	\begin{align*}\label{delta_0}
 	-\Delta P_F\geq C(\Omega )\int_{A'_\epsilon}||\nabla g(x')||^2\;dx'.
 	\end{align*}
 	We point out that, with the last inequality, we have also  proved  that \linebreak $-\Delta P_F\geq 0$.
 \end{proof}

 \begin{lem}
 	Let $\Omega$ be a bounded, open convex set of $\mathbb{R}^n$, then 
 	\begin{equation}\label{deltaM}
 	\Delta W_{F}\leq p\left(r^F_ {\rm max}(\Omega)\right)^{p-1}\Delta V+\left(r^F_ {\rm max}(\Omega )\right)^p\Delta P_F+o(\Delta P_F)+o(\Delta V).
 	\end{equation}
 \end{lem}
 \begin{proof}
 	\begin{align*}
 	&	-\Delta W_F(\Omega)=\int_{C_\epsilon}\left( F^o(x)\right)^p F(\nu_{\de\Omega}(x))\;d\mathcal{H}^{n-1}(x)-\int_{A_{\epsilon}}\left(F^o(x)\right)^pF(\ubar{0},1)\;d\mathcal{H}^{n-1}(x)=\\&=\int_{A'_\epsilon}\left( F^o(x',g(x'))  \right)^p F(-\nabla g(x'),1)\;dt-F(\ubar{0},1) \int_{A'_\epsilon}\left(F^o\left(x',y_0-\epsilon\right)\right)^p\;dx'=\\&
 	=\int_{A'_\epsilon} \left[\left(   F^o(x',g(x'))\right)^p-\left( F^o(x', y_0-\epsilon) \right)^p \right] F(-\nabla g(x'),1)\;dx'+\\&+\int_{A'_\epsilon}\left[ F(-\nabla g(x'),1)-F(\ubar{0},1)  \right] \left(F^o(x',y_0-\epsilon)\right)^p\;dx'=I_1+I_2.
 	\end{align*}
 	Firstly, we take into consideration $I_2$.

 	\underline{Claim $1$}: $$F^o(x',y_0-\epsilon)=r^F_{\max}(\Omega)+o(1),$$ where we use the following notation: $q(\epsilon)=:o(\epsilon^n)$ if $\lim_{\epsilon\rightarrow0}q(\epsilon)/\epsilon^n=0.$
 	
 	Using Taylor 
 	\begin{align*}
 	& F^o(x', y_0-\epsilon)=F^o(x'_0,y_0)+\langle\nabla F^o(x'_0,y_0), (x'-x'_0,-\epsilon)\rangle+o(||(x'-x'_0,-\epsilon)||)=\\ & =r^F_{\max}(\Omega)+\langle\nabla F^o(x'_0,y_0), (x'-x'_0,-\epsilon)\rangle+o(||(x'-x'_0,-\epsilon)||).
 	\end{align*}
 	For the Cauchy-Schwartz inequality:
 	\begin{align*}
 	&	|\langle \nabla F^o(x'_0,y_0), (x'-x_0,-\epsilon)\rangle|\leq ||\nabla F^o(x'_0,y_0)||\sqrt{||x'-x_0||^2+\epsilon^2}\leq\\& \leq||\nabla F^o(x'_0,y_0)||\sqrt{\max_{x'\in A'_\epsilon}\{||x'-x'_0||\}+\epsilon^2}=o(1).
 	\end{align*}
 	So we have the claim.
 	
 	Using Claim $1$,  we have that 
 	\begin{align*}
 	& I_2=\int_{A'_{\epsilon}}\left[ F(-\nabla g(x'),1)-F(\ubar{0},1)  \right] \left( r^F_{\max}(\Omega)+o(1) \right)^p\;dy=\\&=\int_{A'_{\epsilon}}\left[ F(-\nabla g(x'),1)-F(\ubar{0},1)  \right] \left(	\left( r^F_{\max}(\Omega)\right)^p+o(1) \right)\;dy=\left(r^F_{\max}(\Omega)\right)^p\Delta P_F(\Omega)+o(1)\Delta P_F=\\&=\left(r^F_{\max}(\Omega)\right)^p\Delta P_F(\Omega)+o(\Delta P_F).
 	\end{align*}
 	
 	We study now $I_1$.
 	
 	From the convexity inequality we have 
 	\begin{equation*}
 	\left( F^o(x',g(x'))\right)^p- \left(F^o(x',y_0-\epsilon)\right)^p\geq p\left(F^o(x',y_0-\epsilon)\right)^{p-1}\langle\nabla  F^o(x',y_0-\epsilon),\left(\ubar{0},h(x')\right)\rangle.
 	\end{equation*}
 	Using the last convexity inequality we have
 	\begin{align*}
 	& I_1=\int_{A'_\epsilon} \left[\left(   F^o(x',g(x'))\right)^p-\left( F^o(x', y_0-\epsilon) \right)^p \right] F(-\nabla g(x'),1)\;dx'\geq\\ &\geq\int_{A'_{\epsilon}}p \left(F^o(x', y_0-\epsilon) \right)^{p-1}\langle\nabla F^o(x',y_0-\epsilon),(0,h(x'))\rangle F(-\nabla g(x'),1)\;dy=\\&=\int_{A'_{\epsilon}}p\left(F^o(x',y_0-\epsilon)\right)^{p-1}\dfrac{\de F^o}{\de x_n}(x',y_0-\epsilon)  h(x') F(-\nabla g(x'),1)\;dx'.
 	\end{align*}

 	\underline{Claim $2$} 	\begin{equation}
 	\dfrac{\de F^o}{\de x_n}(x', y_0-\epsilon)=\dfrac{F^o(x'_0,y_0)}{(y_0-\epsilon)}+o(1).
 	\end{equation}
 	
 	Using Taylor and the property $\langle\nabla F^o(\xi),\xi\rangle=F^o(\xi)$, we have that 
 	\begin{align*}
 	&  F^o(x'_0,y_0-\epsilon)=F^o(x',y_0-\epsilon)+o(1)=\langle\nabla F^o(x', y_0-\epsilon), (x',y_0-\epsilon)\rangle+o(1)=\\&=\langle\nabla_{x'}F^o(x', y_0-\epsilon),x'\rangle+\left(y_0-\epsilon\right)\dfrac{\de F^o}{\de x_n}(x',y_0-\epsilon)+o(1), 
 	\end{align*}
 	and consequently
 	\begin{equation}\label{remtre}
 	\dfrac{\de F^o}{\de x_n}(y, y_0-\epsilon)=\dfrac{F^o(x'_0,y_0-\epsilon)}{(y_0-\epsilon)}-\dfrac{1}{(y_0-\epsilon)}\langle\nabla_{x'}F^o(x', y_0-\epsilon),x'\rangle+o(1).
 	\end{equation}
 	%	We now prove that 
 	%	\begin{equation*}
 	%	\langle\nabla_{x'}F^o(x', y_0-\epsilon),x'\rangle=o(1).
 	%	\end{equation*}
 	%	We define $s(y):=F^o(y,g(y))$ and we have that $s(x_0)=r^F_{\max}(\Omega)$, that is a global maximum. We have that 
 	%	\begin{align*}
 	%	\nabla_{\bar{x}}s(y)=\nabla_{\bar{x}}F^o(y,g(y))+\dfrac{\partial F^o}{\partial x_n}(y,g(y)) \nabla_{\bar{x}}g(y).
 	%	\end{align*}
 	%	Being $\nabla_{\bar{x}}g(x_0,y_0)=0$, we obtain 
 	%	\begin{equation*}
 	%	\ubar{0}=\nabla_{\bar{x}}s(x_0)=\nabla_{\bar{x}}F^o(x_0,y_0)
 	%	\end{equation*}
 	Considering the fact that  $\nabla F^o(x'_0,y_0)=(\ubar{0},1)$, we have that 
 	\begin{align*}
 	\langle\nabla_{x'}F^o(x', y_0-\epsilon),x'\rangle=\sum_{i=1}^{n-1}x_i\dfrac{\de F^o}{\de x_i}(x',y_0-\epsilon)=\sum_{i=1}^{n-1}x_i\left(\dfrac{\de F^o}{\de x_i}(x'_0,y_0)+o(1)\right)=o(1).
 	\end{align*}
 	So, from \eqref{remtre} and Claim $1$, we obtain the claim
 	\begin{equation*}
 	\dfrac{\de F^o}{\de x_n}(x', y_0-\epsilon)=\dfrac{F^o(x'_0,y_0)}{(y_0-\epsilon)}+o(1).
 	\end{equation*}

 	\underline{Claim $3$} \begin{equation*}
 	F(-\nabla g(x'),1)=F(\ubar{0},1)+o(1).
 	\end{equation*}
 	Using Taylor and the facts that $\nabla g=\nabla h$ is continuous and $\nabla h(x'_0)=0$, we obtain
 	\begin{align*}
 	F(-\nabla h(x'),1)=F(\ubar{0},1)+\langle\nabla F(\ubar{0},1), (-\nabla h(x'),0)\rangle+o(||-\nabla h(x')||)=F(\ubar{0},1)+o(1),
 	\end{align*}

 	Using Claim $1$, Claim $2$ and Claim $3$:
 	\begin{align*}
 	&	I_1\geq \int_{A'_\epsilon}p \left(r^F_{\max}(\Omega)+o(1)\right)^{p-1}\left( \dfrac{F^o(x'_0,y_0)}{(y_0-\epsilon)}+o(1)\right) h(x')\left(F(\ubar{0},1)+o(1)\right)\;dy\\&=\int_{A'_\epsilon}p \left(r^F_{\max}(\Omega)\right)^{p-1}\dfrac{F^o(x'_0,y_0)}{(y_0-\epsilon)} h(y)F(\ubar{0},1)\;dy+o(-\Delta V)\geq \\&\geq  \int_{A'_\epsilon}p \left(r^F_{\max}(\Omega)\right)^{p-1}\dfrac{y_0}{y_0-\epsilon} h(x')\;dx'+o(-\Delta V)\geq p \left(r^F_{\max}(\Omega)\right)^{p-1}(-\Delta V)+o(-\Delta V),
 	\end{align*}
 	where we have used the fact that 
 	\begin{equation*}
 	F(\ubar{0},1) F^o(x'_0,y_0)\geq |\langle(\ubar{0},1), (x'_0,y_0)\rangle|=y_0.
 	\end{equation*}
 \end{proof}

\begin{lem}\label{new_lemma}
	Let $\Omega$ be a bounded, open convex set of $\mathbb{R}^n$. Then,
	\begin{equation}\label{dis_new}
	\dfrac{M_F(\Omega)}{P_F(\Omega)}\leq \left(r^F_{\max}(\Omega)\right)^p
	\end{equation}
	and equality holds if and only if $\Omega$ is a Wulff shape centered at the origin.
\end{lem}
\begin{proof}
	If $\Omega$ is a Wulff shape, then
	\begin{equation}
	\dfrac{M_F(\Omega)}{P_F(\Omega)}=\left(r^F_{\max}	(\Omega)\right)^p\dfrac{P_F(\Omega)}{P_F(\Omega)}=\left(r^F_{\max}(\Omega)\right)^p.
	\end{equation}
	If $\Omega$ is not a Wulff shape, consider the set 
	\begin{equation*}
	S:=\{  x\in\de\Omega\; :\; F^o(x)<r^F_{\max}(\Omega)  \}.
	\end{equation*}
	Since $F^o$ is a continous function,  $\mathcal{H}^{n-1}(S)>0$ and, by definition of $r^F_{\max}(\Omega)$, we have that 
	$$\de\Omega\setminus S=\{x\in\de\Omega\;:\; F^o(x)=r^F_{\max}(\Omega) \} .$$
	Thus, we obtain
	\begin{multline*}
	\dfrac{M_F(\Omega)}{P_F(\Omega)}=\dfrac{\ds\int_{S} [F^o(x) ]^p\;F(\nu_{\partial\Omega}(x))\;d\mathcal{H}^{n-1}(x)+\ds\int_{\partial \Omega\setminus S} [F^o(x) ]^p\;F(\nu_{\partial\Omega}(x))\;d\mathcal{H}^{n-1}(x)}{P_F(\Omega)}\\<\dfrac{\ds\int_{S} \left(r^F_{\max}(\Omega)\right)^p\;F(\nu_{\partial\Omega})\;d\mathcal{H}^{n-1}+\ds\int_{\partial \Omega\setminus S} \left(r^F_{\max}(\Omega)\right)^p\;F(\nu_{\partial\Omega})\;d\mathcal{H}^{n-1}}{P_F(\Omega)}=\left(r^F_{\max}(\Omega)\right)^p
	\end{multline*}
\end{proof}
\begin{prop}
	Let $\Omega$ be a bounded, open convex set of $\mathbb{R}^n$ such that
	\begin{equation}\label{pos_excess}
	E_F(\Omega)>0,
	\end{equation} 
	then $\Omega$ is not a minimizer of $\mathcal{F}(\cdot)$.
\end{prop}
\begin{proof}

	Using $\eqref{deltaM}$, we have that
	\begin{align}\label{delta_F}&\Delta \mathcal{F} =\dfrac{1}{V(\Omega)^{\frac{p}{n}}P_F(\Omega)}\left(\Delta M_{F}-\dfrac{\Delta P_F}{P_F(\Omega)}M_{F}(\Omega)-\frac{p}{n}\dfrac{\Delta V}{V(\Omega)}M_{F}(\Omega)\right)+o(\Delta P_F)+o(\Delta V)=\\\nonumber&\leq\dfrac{1}{V(\Omega)^{\frac{p }{n}}P_F(\Omega)}\left[p\left(\left(r^F_{\max}(\Omega)\right)^{p-1}-\frac{M_{F}(\Omega)}{nV(\Omega)}\right)\Delta V+   \right.\\\nonumber&\ \ \ \ \ \ \ \ \ \ \ \ \ \ \ \ \ \ \ \ \ \ \ \ \ \ \ \ \ \ \left. \left((r^F_{\max}(\Omega))^p-\frac{M_{F}(\Omega)}{P_F(\Omega)}\right) \Delta P_F   \right]+o(\Delta P_F)+o(\Delta V)=\\\nonumber&=\dfrac{1}{V(\Omega)^{\frac{p}{n}}P_F(\Omega)}\left[pE_F(\Omega)\Delta V+\left((r^F_{\max}(\Omega))^p-\frac{M_F(\Omega)}{P_F(\Omega)}\right) \Delta P_F   \right]+o(\Delta P_F)+o(\Delta V)
	\end{align}
	Since \eqref{pos_excess} holds, $\Omega$ cannot be a ball centered at the origin.
	From Lemma \ref{new_lemma},  follows  that 
	$$ (r^F_{\max}(\Omega))^p-\frac{M_F(\Omega)}{P_F(\Omega)}>0.$$
	Considering also that $\Delta V<0$ and $\Delta P_F<0$, we can conclude that 
	$$ \Delta \mathcal{F}<0.$$
	
\end{proof}
\subsection{Wulff shapes are the unique minimizers having vanishing Excess}
\begin{prop}
	Let $\Omega$ be a bounded, open convex set of $\mathbb{R}^n$ such that
	\begin{equation}\label{zero_excess}
	E_F(\Omega)=0,
	\end{equation}
	then either $\Omega $ is the Wulff shape centered at the origin or it is not a minimizer of $\mathcal{F}(\cdot)$. 
\end{prop}
\begin{proof}
	From \eqref{inequality}, \eqref{zero_excess}, \eqref{delta_F}, we obtain the following expression
	$$ \Delta \mathcal{F}(\Omega)=\dfrac{1}{V(\Omega)^{\frac{p}{n}}P_F(\Omega)}\left[\left((r^F_{\max}(\Omega))^p-\frac{M_F(\Omega)}{P_F(\Omega)}\right) \Delta P_F   \right]+o(\Delta P_F).$$
	If $$ (r^F_{\max}(\Omega))^p=\frac{M_F(\Omega)}{P_F(\Omega)},$$
	then $\Omega$ is a Wulff shape centered at the origin. If  $ \Delta \mathcal{F}<0 $, then $\Omega$ is not a minimizer. Thus, we have proved the desired claim.
\end{proof}


\begin{thebibliography}{99}
	\bibitem{aflt} A. Alvino, V. Ferone, P.-L. Lions, G. Trombetti, \emph{Convex symmetrization and applications}, Ann. Inst. H. Poincarè Anal. Non Lin\'{e}aire 14 (1997) no. 2, 275-293.
	\bibitem{b} G. Bellettini, \emph{Lecture Notes on Mean Curvature Flow: Barriers and Singular Perturbations}, Lecture Notes Scuola Normale Superiore 2013.
	\bibitem{bpa}G. Bellettini, M. Paolini, \emph{Anisotropic motion by mean curvature in the context of Finsler geometry}, Hokkaido Math. J. 25 (1996) 537-566.
	\bibitem{brn}G. Bellettini, M. Novaga, G. Riey, \emph{First variation of anisotropic energies and crystalline mean curvature for partitions}, Interfaces and Free Bound. 5 (2003) no. 3, 331-356.
	\bibitem{bbmp} M.F. Betta, F. Brock, A. Mercaldo, M.R. Posteraro, \emph{Weighted isoperimetric inequalities on $\mathbb{R}^n$ 
		and applications to rearrangements}, Math. Nachr. 281 (2008) no. 4, 466-498.
	\bibitem{bf} L. Brasco, G. Franzina, \emph{An anisotropic eigenvalue problem of Stekloff type and weighted Wulff inequalities}, Nonlinear Diff. Eq. Appl. 20 (2013) no. 6, 1795-1830.
	\bibitem{br} F. Brock, \emph{An isoperimetric inequality for eigenvalues of the Stekloff problem}, ZAMM Z. Angew. Math. Mech.  81 (2001) no. 1, 69-71.
	\bibitem{bfnt} D. Bucur, V. Ferone, C. Nitsch, C. Trombetti, \emph{Weinstock inequality in higher dimensions}, arXiv:1710.04587v2.
	\bibitem{bu}H. Busemann, \emph{The isoperimetric problem for Minkowski area}, Amer. J. Math. 71 (1949) 743-762.
	\bibitem{dp} B. Dacorogna, C.E. Pfister, \emph{Wulff theorem and best constant in Sobolev inequality}, J. Math. Pures  Appl. (9) 71 (1992) no. 2, 97-118.
	\bibitem{dpg} F. Della Pietra, N. Gavitone, \emph{Symmetrization with respect to the anisotropic perimeter and applications}, Math. Ann. 363 (2015) no. 3-4, 953-971.
	\bibitem{dpgp} F. Della Pietra, N. Gavitone, G. Piscitelli, \emph{A sharp weighted anisotropic Poincar\'{e} inequality for convex domains}, C. R. Math. Acad. Sci. Paris 355 (2017) no. 7, 748-752.
	\bibitem{dpgx} F. Della Pietra, N. Gavitone, C. Xia, \emph{Motion of level sets by inverse anisotropic mean curvature}, arXiv:1804.06639.x
	\bibitem{fm} I. Fonseca, S. M\"{u}ller, \emph{A uniqueness proof for the Wulff theorem}, Proc. Roy. Soc. Edinburgh Sect. A 119 (1991) no. 1-2, 125-136.
	\bibitem{h} A. Henrot, \emph{Extremum problems for eigenvalues of elliptic operators}, Frontiers in Mathematics, Birkh\"{a}user Verlag, Basel, 2006.
	\bibitem{hp} A. Henrot, M. Pierre, \emph{Variation et optimisation de formes. Une analyse g\'{e}om\'{e}trique}, Math\'{e}matiques \& Applications, Birkh\"{a}user Verlag, Basel, 2006.
	\bibitem{hu} G. Huisken, T. Ilmanen, \emph{The inverse mean curvature flow and the Riemannian Penrose inequality}, J. Differential Geom. 59 (2001) no. 3, 353-437.
	\bibitem{m} F. Maggi, \emph{Sets of Finite Perimeter and Geometric Variational Problems: An Introduction to Geometric Measure Theory}, Cambridge Studies in Advanced Mathematics, 2012.
	\bibitem{r} A. Ros, \emph{Compact hypersurfaces with costant higher order mean curvature}, Rev. Mat. Iberoamericana 3 (1987) no. 3-4, 447-453.
	\bibitem{s} R. Schneider, \emph{Convex  bodies: the Brunn-Minkowski theory}, Cambridge University Press, 1993.
	\bibitem{x} C. Xia, \emph{Inverse anisotropic mean curvature flow and a Minkowski type inequality}, Adv. Math. 315 (2017) 102-129. 
	\bibitem{xz} C. Xia, X. Zhang, \emph{ABP estimate and geometric inequalities}, Comm. Anal. Geom. 25 (2016) no. 3, 685-708.
	
\end{thebibliography}
\end{document}